\documentclass[12pt,twoside]{amsart}
\usepackage{amssymb}
\usepackage{a4wide}
\usepackage[latin1]{inputenc}
\usepackage[T1]{fontenc}
\usepackage{times}

\def\hpq0{h^{p,q}_{\leq 0}}
\def\Hpq0{\H_{\leq 0}^{p,q}}

\def\R{{\mathbb R}}

\def\C{{\mathbb C}}

\def\H{{\mathcal H}}

\def\Re{{\rm Re\,  }}

\def\be{\begin{equation}}
\def\ee{\end{equation}}

\newtheorem{thm}{Theorem}[section]
\newtheorem{lma}[thm]{Lemma}

\theoremstyle{definition}

\theoremstyle{remark}

\newtheorem{preremark}{Remark}
\newtheorem{preex}{Example}

\numberwithin{equation}{section}

\title[]
{Complex  integrals  and  Kuperberg's proof of the Bourgain-Milman theorem.}

\address{Department of Mathematics\\Chalmers University
  of Technology \\
 S-412 96
  G\"OTEBORG\\SWEDEN} 

\email{ bob@chalmers.se}

\author[]{Bo Berndtsson}

\begin{document}
\begin{abstract} We give a proof of the Bourgain-Milman theorem using complex methods. The proof is inspired by Kuperberg's, but  considerably shorter.
\end{abstract}

\maketitle

\section{introduction}

If $K$ is a convex body in $\R^n$, its polar body is
$$
K^\circ=\{ \xi\in \R^n; x\cdot \xi\leq 1,    \, \text{ if} \,\, x\in K\}.
$$
The Mahler volume of $K$ is
$$
M(K)=|K||K^\circ|,
$$
where $|K|$ is the Lebesgue volume of $K$.

The Bourgain-Milman theorem (\cite{Bourgain-Milman}) says that there is a universal constant, $C$, not depending on the dimension, such that
$$
M(K)\geq C^n/n!.
$$
Mahler's conjecture says that if $K$ is symmetric ($-K=K$),  $C$ can be taken equal to 4, which is what you get when $K$ is a cube. 

There are by now several proofs of the Bourgain-Milman inequality, including \cite{Nazarov} by Nazarov, \cite{Kuperberg}, \cite{Kuperberg2}
by Kuperberg and  \cite{Giannopoulos}, by Giannopoulos, Paouris and Vritsiou. All of these proofs are very  interesting; Kuperberg's gives the so far best known constant, $C=\pi$ in the symmetric case.
In this note we will give a variant of Kuperberg's proof which replaces his use of 'Gauss linking integrals' by a complex analytic argument from \cite{B}. The main novelty here is the combination of the methods of \cite{B} with Kuperberg's ideas, but no knowledge of the results of \cite{B} is necesseray to read this paper. In a companion paper (\cite{Berndtsson}) we will also give some remarks and variations of Nazarov's proof. 

In the next section we give the proof of Kuperberg's estimate, or actually a function version of his estimate, involving Legendre transforms of convex functions instead of polars of convex bodies (Theorem 2.1). The relation between the function version and estimates for the Mahler volume of convex bodies is discussed in Remark 3, section 2. The proof looks formally quite different from Kuperberg's so in a last section we try to explain the points relating the two proofs.

Finally I would like to thank Bo'az Klartag and Yanir Rubinstein for very stimulating discussions on these matters and two anonymous referees for useful comments and corrections.  

\section{Kuperberg's version of the Bourgain-Milman theorem}
Let $\phi(x)$ be a strictly convex and smooth function on $\R^n$. Its Legendre transform is 
$$
\phi^*(\xi)=\sup_x \xi\cdot x -\phi(x).
$$
For later use we remark  that the supremum is attained when the gradient of the RHS vanishes, i.e. when
$$
\xi=\partial\phi/\partial x.
$$
Hence
\be
\phi^*(\xi)+\phi(x)= \xi\cdot x,
\ee
when $\xi=\partial\phi/\partial x$.  Our main result is as follows:
\begin{thm} If $\phi$ is an even convex function on $\R^n$, then 
$$
\int_{\R^n} e^{-\phi}\int_{\R^n}e^{-\phi^*}\geq \pi^n.
$$
\end{thm}

For the proof, we assume first that $\phi$ is smooth and strictly convex. That this is no serious restriction follows from Lemma 2.4 below.  Let 
$$
\Lambda= \{(x,y,\xi, \eta); \xi=\partial\phi(x)/\partial x, \eta=\partial\phi(y)/\partial y\} \subset \R^n_x\times\R^n_y\times\R^n_\xi\times\R^n_\eta.
$$
We now define a map $\pi$ from $\Lambda$ to $\R^n_t\times\R^n_s$ by
$$
t=\frac{x+y}{2}, s=\frac{\xi-\eta}{2}.
$$
\begin{lma} If $\phi$ is smooth and strictly convex the map $\pi$ is injective from $\Lambda$ to $\R^{2n}$. It is surjective if $\phi$ grows faster than any linear function at infinity, i. e. if for any $C$, $\phi(x)\geq C|x|$ for $|x|$ large enough. 
\end{lma}
\begin{proof}
The  statements mean  that for any $t$, the map 
$$
(x,y)\to \partial\phi(x)/\partial x -\partial\phi(y)/\partial y=: p_t(x,y)
$$
 from $\{x+y=2t\}$ to $\R^n$ is injective and surjective respectively. Let $\Phi(x):=\phi(x)+\phi(2t-x)$.
Then $p_t(x,y)=\partial\Phi(x)/\partial x$. Since $\Phi$ is smooth and strictly convex, this map is injective. For the second part we have that $\Phi$ grows faster than any linear function if $\phi$ does. Hence $\Phi(x) -x\cdot s$ has a minimum for any $s$ in $\R^n$, and $\partial\Phi(x)/\partial x=s$ there.  Hence, $p_t$ is surjective then.
\end{proof}

We next pull back the Mahler integral 
$$
\int_{\R^n_t\times\R^n_s} e^{-(\phi(t)+\phi^*(s))} dtds
$$
to $\Lambda$ by $\pi$.
To compute the pull-back we introduce complex notation, $z=x+iy$, $\zeta=\xi+i\eta$, and  the differential forms
$$
\omega=\frac{i}{2}\sum dz_j\wedge d\bar\zeta_j,  \,\, \Omega=\omega^n/n!= a_n (\frac{i}{2})^n dz\wedge d\bar\zeta,
$$
where $dz:=dz_1\wedge ...\wedge dz_n$, and $a_n=(-1)^{n(n-1)/2}$. Both forms $\omega$ and $\Omega$ are closed and, moreover, $\Omega$ remains closed after multiplication by any holomorphic function of $z, \bar\zeta$. This will play an important role in the sequel.

In the next lemma we identify the Lebesgue volume form $dtds$ on $\R^n_t\times\R^n_s$ with the differential form $dt\wedge ds:=dt_1\wedge ...\wedge dt_n\wedge ds_1\wedge ...\wedge ds_n$. 
\begin{lma} 
$$
\pi^*(dtds)=(- 2)^{-n}a_n \Omega.
$$
\end{lma}
\begin{proof}
Let
$$
\tau:=\sum dt_j\wedge ds_j.
$$
Then $dtds=a_n\tau^n/n!$. By the definition of $\pi$
$$
\pi^*(\tau)=\frac{1}{4}\sum (dx_j\wedge d\xi_j -dx_j\wedge d\eta_j+ dy_j\wedge d\xi_j -dy_j\wedge d\eta_j).
$$
On $\Lambda$, $\sum\xi_j dx_j=d\phi(x)$. Taking the exterior derivative we find that $\sum dx_j\wedge d\xi_j=0$ on $\Lambda$. In the same way, $\sum dy_j\wedge d\eta_j=0$ on $\Lambda$. Hence
$$
\pi^*(\tau)=\frac{1}{4}\sum (-dx_j\wedge d\eta_j+ dy_j\wedge d\xi_j ).
$$
On the other hand
$$
\omega=\frac{i}{2}\sum (dx_j\wedge d\xi_j+dy_j\wedge d\eta_j +i(-dx_j\wedge d\eta_j+ dy_j\wedge d\xi_j)).
$$
As we have just seen the real part of the sum vanishes on $\Lambda$. Hence
$$
\omega=\frac{-1}{2}\sum (-dx_j\wedge d\eta_j+ dy_j\wedge d\xi_j)= (-2)\pi^*(\tau).
$$
Taking the $n^{th}$ exterior power of both sides and dividing by $n!$ the lemma follows.
\end{proof}

As for the integrand in the Mahler integral we first note that by convexity, if $(t,s)=\pi(x,y,\xi,\eta)$, 
$$
\phi(t)=\phi((x+y)/2)\leq (\phi(x)+\phi(y))/2.
$$
Similarily
$$
\phi^*(s)=\phi^*((\xi-\eta)/2)\leq (\phi^*(\xi)+\phi^*(-\eta))/2=(\phi^*(\xi)+\phi^*(\eta))/2,
$$
where the assumption that $\phi$, and therefore $\phi^*$, is even is used in the last equality. Summing these two inequalities we get
$$
\phi(t)+\phi^*(s)\leq (1/2)\left(\phi(x)+\phi^*(\xi) +\phi(y)+\phi^*(\eta)\right).
$$
Invoking (2.1) we see that on $\Lambda$
$$
\phi(t)+\phi^*(s)\leq (1/2)(x\cdot\xi+y\cdot\eta),
$$
so
\be
 e^{-(1/2)(x\cdot\xi+y\cdot\eta)}\leq \pi^*(e^{-\phi -\phi^*}).
\ee

We now combine (2.2) and Lemma 2.3 and get the following lower bound for the Mahler integral
\be
2^{-n} \left |\int_\Lambda e^{-(1/2)(z\cdot\bar\zeta)} \Omega \right | \leq 
2^{-n}\int_\Lambda e^{-(1/2)(x\cdot\xi+y\cdot\eta)} |\Omega|\leq \int_{\R^n_t\times\R^n_s} e^{-(\phi(t)+\phi^*(s))} dtds.
\ee
The rest of the argument is basically that we  deform $\Lambda$ continuously to another 'contour'. This should not change the integral
\be
\int_\Lambda e^{-(1/2)(z\cdot\bar\zeta)} \Omega,
\ee
since, as remarked above,  the integrand is a closed form. The only slight complication here is that $\Lambda$ is unbounded and we need to estimate the tails; one possible way to handle this is given below (there are many others).

Recall that $\Lambda$ is given by the equations $\xi=\partial\phi(x)/\partial x$ and $\eta=\partial\phi(y)/\partial y$. Now, say that we can deform $\Lambda=\Lambda_\phi$ to the manifold $\Lambda_0$, defined by $\phi(x)=|x|^2/2$ without changing the integral (2.4). 
Then $\Lambda_0$ is defined by  $\xi=x, \eta=y$, so $\Lambda_0$ is the diagonal in $\C^n_z\times\C^n_\zeta$. 
Hence we see that  the left hand side of (2.3) equals
$$
2^{-n}\int_{\C^n} e^{-|z|^2/2} dm=2^{-n}(\int_\C e^{-|z_1|^2/2})^n= \pi^{n}.
$$
Accepting the deformation argument above this completes the proof of the theorem. 

We now turn to the rigorous verification. We first give an elementary lemma that must be well known.
\begin{lma} Let $\phi_j$ be a sequence of convex functions decreasing to $\phi$ as $j\to \infty$. Then the sequence $\phi_j^*$ increases to $\phi^*$. Similarily, if $\phi_j$ increases to $\phi$, then $\phi_j^*$ decreases to $\phi^*$.
\end{lma}
\begin{proof}
In the proof we will allow also convex functions that attain the value $+\infty$, cf. \cite{Rockafellar}.
Assume first that $\phi_j$ decreases. Take $\delta>0$. For $x$   fixed such that $\phi(x)<\infty$, take $j$ so large that $\phi_j(x)<\phi(x)+\delta$. Then, for any $\xi$, 
$$
x\cdot\xi-\phi(x)\leq x\cdot\xi-\phi_j(x)+\delta\leq \phi_j^*(\xi)+\delta\leq \lim\phi_j^*(\xi)+\delta.
$$
Clearly this inequality holds also if $\phi(x)=\infty$. 
Taking the supremum over all $x$  we see that  $\phi^*(\xi)\leq\lim\phi_j^*(\xi)$. This  proves the first claim since the opposite inequality is evident. 

Now assume that $\phi_j$ increases to $\phi$. The $\phi_j^*$ decreases to a limit that can be written $\phi_\infty^*$ for some convex function $\phi_\infty$. Taking Legendre transforms, we have by the first part that $\phi_j$ increases to $\phi_\infty$. Hence, 
$\phi_\infty=\phi$, which completes the proof.
\end{proof}
From the lemma we see, by monotone convergence,  that if Theorem 2.1 holds for a monotone sequence of convex functions $\phi_j$, then it holds for the limit function as well. To use this, we first note that the deformation argument above goes through if $\phi$ is smooth and $\phi(x)=|x|^2/2 +C$ for $|x|$ sufficiently large, since in that case we have good decay of the integrand at infinity.  This is much like deforming the contour using Cauchy's integral theorem in one complex variable and we next give the details of this argument. 

Recall that for a smooth convex function $\phi$ on $\R^n$ 
$$
\Lambda=\Lambda_\phi= \{(x,y,\xi, \eta); \xi=\partial\phi(x)/\partial x, \eta=\partial\phi(y)/\partial y\} \subset \R^n_x\times\R^n_y\times\R^n_\xi\times\R^n_\eta,
$$
and denote the integral
$$
\int_{\Lambda _\phi}e^{-(1/2)(z\cdot\bar\zeta)} \Omega =: I_\phi.
$$
Our claim is that if $\phi_0=|x|^2/2$ and $\phi_1=|x|^2/2 +C$ for $|x|>R$, then $I_{\phi_0}=I_{\phi_1}$. We may assume that $\phi_1(0)=0$. 

Let, for $0\leq t\leq 1$, $\phi_t=t\phi_1+(1-t)\phi_0$. Define, for $z=x+iy\in\C^n$ and $0\leq t\leq 1$, $\phi_t=t\phi_1+(1-t)\phi_0$, 
$$
H_t(z)=(x,y, \partial\phi_t(x)/\partial x, \partial\phi_t(y)/\partial y).
$$
$H_t$ gives a homotopy between the 'contours' $\Lambda_{\phi_0}$ and $\Lambda_{\phi_1}$. Put
$$
\alpha=e^{-(1/2)(z\cdot\bar\zeta)} \Omega .
$$
Then
$$
I_{\phi_0}-I_{\phi_1}= \int_{\C^n} H_0^*(\alpha)-\int_{\C^n}H_1^*(\alpha).
$$
Recall that on $\Lambda_{\phi_t}$,
$$
|e^{-1/2(z\cdot\bar\zeta)}|=e^{-(1/2)(\phi_t(x)+\phi_t(y)+\phi_t^*(\xi)+\phi_t^*(\eta))}\leq e^{-(1/2)(\phi_t(x)+\phi_t(y))},
$$
since $\phi_t^*(\xi)\geq -\phi_t(0)=0$.  If $|z|>2R$ and , say, $|y|\geq|x|$, then $|y|>R$ so $\phi_t(y)=|y|^2+tC\geq |z|^2/2+tC$. Hence
$$
|e^{-1/2(z\cdot\bar\zeta)}|\leq C_1e^{-|z|^2/4}
$$
on $\Lambda_{\phi_t}$. Moreover, the second derivatives of $\phi_t$ are bounded, and the first derivatives are of at most linear growth. Hence the form $H_t^*(\alpha)$ has coefficients decaying superexponentially when $z$ tends to infinity, both when we take $t$ fixed and when we regard it as a form on $\C^n\times [0,1]$. 

Let $\chi_r(z)=\chi(|z|/r)$, where $\chi$ is compactly supported and equal to 1 for $|z|<1$. Then, for large $r$,
$$
\int_{\C^n} H_0^*(\alpha)-\int_{\C^n}H_1^*(\alpha)=\int_{\C^n} \chi_r H_0^*(\alpha)-\int_{\C^n}\chi_rH_1^*(\alpha)
+\epsilon(r),
$$
where $\epsilon(r)$ tends to zero as $r$ tends to infinity. By Stokes' theorem this equals, since $\alpha$ is closed, 
$$
\int_{\C^n\times [0,1]} d\chi_r\wedge H^*(\alpha) +\epsilon(r).
$$
When $r$ tends to infinity, the integral tends to zero, so $I_{\phi_0}-I_{\phi_1}=0$, which was the claim.

Hence Theorem 2.1 holds for  smooth convex functions $\phi$ that equal $|x|^2+C$ for $|x|$ sufficiently large. Next, this implies that it holds for all  such convex functions, even if they are not smooth. To see this we  just need to regularize $\phi$ by convolving with an approximate identity $\epsilon^{-n}\alpha(|x|/\epsilon)$.  This produces a sequence of smooth functions that decreases to $\phi$ as $\epsilon$ tends to zero (since $\phi$ is in particular subharmonic) and it does not destroy the property of being equal to $|x|^2+C$  for $|x|$ large. Hence Theorem 2.1 holds in the limit as well by the lemma.

 From there it follows that the theorem holds if $\phi$ has at most linear growth, since $\phi$ is then the decreasing limit of 
$$
\phi_j(x):=\max(\phi(x), |x|^2/2-j),
$$
and these functions equal $|x|^2/2-j$ when $|x|$ is large. Finally, an arbitrary $\phi$ is the increasing limit of functions of linear growth, e.g.
$$
\psi_j(x)=\sup_{|\xi|<j} x\cdot\xi-\phi^*(\xi).
$$
This proves Theorem 2.1 in general.

{\bf Remark 1:}
Note that on $\Lambda$, since $\xi=\partial\phi(x)/\partial x$ and $\eta=\partial\phi(y)/\partial y$, the proof of Lemma 2.3 shows that
$$
\omega=\frac{1}{2}\sum (\phi_{j k}(x)  +\phi_{j k}(y)) dx_j\wedge dy_k.
$$
Thus,  if we parametrize $\Lambda$ by $(x,y)$, $\Omega$ becomes a mixed Monge-Ampere form of $\phi(x)$ and $\phi(y)$, whereas (surpringly!) if we parametrize by $(t,s)$ it is, up to a constant, just the Lebesgue volume form $dtds$. 

The proof of the lemma actually shows something stronger than this, namely that $\pi^*(\tau)=(-1/2)\omega$ on $\Lambda$. Hence $\Lambda$ is Lagrangian for the imaginary part of $\omega$, symplectic for the real part, and $(t,s)$ are Darboux coordinates for $\Re(\omega)$. (Meaning that $(t,s)$ reduces the symplectic form $\Re(\omega)$ to the standard symplectic form on $\R^{2n}$.) 
\qed

{\bf Remark 2:} If we make the change of variables
$$
z'= \frac{(1-i)}{2}z, \quad \zeta'=\frac{(1+i)}{2}\zeta,
$$
the form $\omega$ transforms to $-(i/2)\omega$, so  essentially  we switch the real and imaginary parts. Moreover, in these coordinates, the map $\pi$ simply becomes 
$$
t=x'=\Re z', \quad s=\xi'=\Re \zeta'.
$$
\qed

{\bf Remark 3:} It is easy to see that Kuperberg's theorem follows from Theorem 2.1, by applying the theorem to a 1-homogenous function $\phi$; see e. g. \cite{Berndtsson} for details of this argument. The converse direction seems to be a bit more involved. Using partly ideas from \cite{AKM}, it was proved by Fradelezi and Meyer, \cite{Fradelezi} that if Mahler's conjecture does hold, i. e. if
$$
M(K)\geq 4^n/n!
$$
for any symmetric convex body, then Theorem 2.1 follows, with $\pi^n$ replaced by $4^n$. It is however crucial for their argument that one has precisely the constant $4^n$ here. On the other hand, using the Bourgain-Milman inequality,  it was proved by Klartag and Milman, \cite{Klartag-Milman}, that, unconditionally,  Theorem 2.1 holds with $\pi^n$ replaced by $c^n$ for  {\it some}
universal constant $c$.
\section{Comparison with Kuperberg's proof}

Kuperberg considers a convex body $K$, symmetric around the origin,which can be assumed to be strictly convex and smoothly bounded. He then defines
$$
K^+:=\{(x,\xi)\in \partial K\times \partial K^\circ; x\cdot\xi=1\}.
$$
 When $K$ is strictly convex and smoothly bounded, there is for each $x$ in the boundary of $K$ a unique $\xi=\xi(x)$ in the boundary of $K^\circ$ such that $x\cdot \xi(x)=1$, so $K^+$ is the graph of this map. (Hence, in particular, $K^+$ is a smooth manifold.) Concretely, if $\mu$ is the Minkowski functional of $K$ ; $\xi(x)=\partial\mu/\partial x$ (as follows e.g. from Euler's formula applied to the 1-homogenous function $\mu$).

In general, if $\lambda$ is an $n$-dimensional submanifold of $\R^n\times\R^n$, one defines its 'directed volume' in the following way. Parametrize $\lambda$ by a map $F:U\to \R^n\times \R^n$, where $U$ is an open subset of $\R^n$. Write
$$
\vec{F}=\sum_1^n F_j e_j +\sum_1^{n} F_{j+n}f_j,
$$
where $e_j$ is a basis for the first copy of $\R^n$, and $f_j$ is a basis for the second copy. Then the directed volume of $\lambda$ is
$$
\vec{V}(\lambda)=\int_U (d\vec{F})^n/n!.
$$
This is  an $n$-vector in $\bigwedge^n(\R^n\oplus\R^n)$, and the usual change of variables formula shows that it does not depend on the choice of $F$. Since the integrand is exact, the directed volume depends only on the boundary of $\lambda$. The directed volume of the boundary is finally defined as the directed volume of $\lambda$. The main object of Kuperberg's argument is the directed volume of $K^+$. To compute it, we may choose $\lambda$ to be any manifold of the form
$$
\lambda=\{(x,\xi); x\in K; \,\xi=\partial\phi(x)/\partial x\},
$$
where $\phi$ is any convex function equal to $\mu$ near the boundary of $K$. 

Let now 
$$
K^-:=\{ (y,-\eta); (y,\eta)\in K^+\}
$$
and form
$$
V:=\vec{V}(K^+)\wedge\vec{V}(K^-).
$$

This is an element in the top exterior product of $\R^n\oplus\R^n$ so we may consider it as a scalar (after dividing by $e\wedge f$).

Kuperberg's proof now consist of two parts: First, he proves in \cite{Kuperberg} that $V$ equals the volume of a subset of the convex hull of the product of $K^+$ and $K^-$. Since this set is included in $K\times K^\circ$, it follows that $V$ is bounded from above by the Mahler volume of $K$.  The second part is to prove that $V$ is greater than $\pi^n/n!$. This is (roughly) Kuperberg's 'bottle-neck conjecture', proved in the later paper \cite{Kuperberg2}.

The link between this proof and the one we have given in the previous section is that, with $\Lambda= \lambda\times \lambda$ (which is a manifold of  the type we considered in the proof of Theorem 2.1), 
$$
V=\int_{\Lambda} \theta
$$
where $\theta$ is a certain differential form of degree $2n$. The main observation is that on $\Lambda$, $\theta$ equals the differential form that appeared in the previous section and could be viewed either as $\pi^*(dtds)$ or  $2^{-n}\Omega$ restricted to $\Lambda_K$. (This gives a third interesting interpretation of the restriction of $\Omega$ to $\Lambda$; cf. the remark after Lemma 2.3.) This  claim follows from a direct computation   that we omit.  Therefore
\be
V=2^{-n}\int_\Lambda \Omega.
\ee

This observation was a main motivation for our proof, since it shows that it must be possible to bound the integral of $\Omega$ from above by the Mahler volume, by the first part of Kuperberg's proof. Finally, Kuperberg's convex hull of $K^+\times K^-$ was part of the inspiration to consider the map $(x,y,\xi,\eta)\to((x+y)/2,(\xi-\eta)/2)$ in our proof.

\end{document}